\documentclass{article}
\usepackage[a4paper]{geometry}

\usepackage{amsthm}
\usepackage{amsopn}
\usepackage{amsmath}
\usepackage{bbm}
\usepackage{amssymb,mathrsfs}

\usepackage[colorlinks=true,linkcolor=blue]{hyperref}
\usepackage{enumitem}

\DeclareMathOperator{\Hom}{Hom}

\DeclareMathOperator{\Spec}{Spec}

\DeclareMathOperator{\Sym}{Sym}

\def\cA{\mathscr{A}}

\def\cH{\mathscr{H}}

\def\cJ{\mathscr{J}}

\def\cL{\mathscr{L}}

\def\cO{\mathscr{O}}

\def\fg{\mathfrak{g}}

\swapnumbers
\newtheoremstyle{exps}{\topsep}{\topsep}{}{0pt}{\bfseries}{.}{0pt}{}

\newtheorem*{thm*}{Theorem}
\newtheorem*{prop*}{Proposition}
\newtheorem*{lem*}{Lemma}
\newtheorem*{cor*}{Corollary}
\newtheorem*{rem*}{Remark}
\newtheorem{thm}{Theorem}[section]
\newtheorem{prop}[thm]{Proposition}

\newtheorem{ex}[thm]{Example}

\theoremstyle{definition}
\newtheorem*{defn*}{Definition}
\newtheorem*{exer*}{Exercise}
\newtheorem*{ex*}{Example}

\newtheorem*{problem*}{Problem}
\newtheorem{rem}[thm]{Remark}
\newtheorem{nolabel}[thm]{ }

\theoremstyle{exps}

\numberwithin{equation}{thm}

\setenumerate[1]{label={\alph*)}} 

\usepackage{slashed}
\usepackage{authblk}

\def\susy{\slashed{D}}

\DeclareMathOperator{\maps}{Map}
\def\fl{\mathfrak{l}}
\def\Xe{{{X\times T^2}}}
\def\brXe{{{(X\times T^2)}}}
\def\fge{{{\fg\oplus \mathfrak t^2}}}
\def\brfge{{{(\fg\oplus \mathfrak t^2)}}}
\def\Ge{{{G\times \mathbb R^2}}}
\def\brGe{{{(G\times \mathbb R^2)}}}
\def\cLe{{{\cL\oplus \cL_{T^2}}}}
\def\brcLe{{{(\cL\oplus \cL_{T^2})}}}
\def\cHe{{{\cH\otimes \cH_{T^2}}}}
\def\Ve{{{V\oplus \mathbb R}}}
\def\brVe{{{(V\oplus \mathbb R)}}}
\def\Lambdae{{{\Lambda\oplus \mathbb Z}}}
\def\brLambdae{{{(\Lambda\oplus \mathbb Z)}}}
\def\Gammae{{{\Gamma\oplus \mathbb Z^2}}}
\def\brGammae{{{(\Gamma\oplus \mathbb Z^2)}}}
\def\vac{|0\rangle}
\def\cV{\mathcal{V}}
\def\super{{\rm super}}
\begin{document}
\title{On a complex-symplectic mirror pair}
\author[1]{Aldi, M}
\author[2]{Heluani, R.} 
\affil[1]{Virginia Commonwealth University}
\affil[2]{IMPA}
\date{}
\maketitle

\abstract{We study the canonical Poisson structure on the loop space of the super-double-twisted-torus and its quantization. As a consequence we obtain a rigorous construction of mirror symmetry as an intertwiner of the N=2 super-conformal structures on the super-symmetric sigma-models on the Kodaira-Thurston nilmanifold and a gerby torus of complex dimension 2. As an application we are able to identify global moduli of equivariant generalized complex structures on these target spaces with moduli of equivariant orthogonal complex structures on the doubled geometry.}

\section{Introduction}
\begin{nolabel}\label{no:first-section}
The full (including all winding sectors) bosonic sigma-model on the 3-dimensional Heisenberg nilmanifold was quantized in \cite{heluanaldi} by combining Hamiltonian quantization,  double field theory (DFT)\cite{hull-zwiebach} and a detailed analysis of the loop space of the so-called double twisted torus \cite{hull-reid-edwards}. The latter is a $T^3$-fibration over $T^3$ which makes the T-duality between the Heisenberg nilmanifold and a $3$-torus $T^3$ twisted by a 3-form flux a manifest symmetry of the theory.

In this paper we continue the work of \cite{heluanaldi} by including fermionic fields to obtain a super-symmetric quantum theory. Just as dilogarithmic singularities in the OPEs prevent the full sigma-model with target the 3-dimensional Heisenberg nilmanifold from being encoded by a vertex algebra, the super-symmetric quantum theory cannot be fully formulated in the language of vertex super-algebras. These systems have been actively studied in relations to both vertex operator algebras and dilogarithmic singularities in the last few years, see \cite{lust, szabo, mathai-linshaw, mathai-hekmati} and references therein.

To exploit the well-known correspondence between super-symmetric sigma-models and generalized complex structures on even-dimensional targets \cite{roceck-zabzine, gates-hull-roceck}, we take direct products of the T-dual targets considered in \cite{heluanaldi} with a circle. As a result, we consider a T-dual pair $(Y,N)$ where $Y$ is a 4-torus twisted by a 3-form and $N$ is a 4-dimensional nilmanifold $N$ with the topology of a $S^1$-fibration over $T^3$ (or equivalently of a $T^2$-fibration of a $T^2$-base). While it is well-known that $N$ admits both symplectic and complex structures, it admits no K\"ahler structure since its first Betti number is odd \cite{thur}. 

One of the goals of this paper is to understand, at the level of the full super-symmetric quantum theory, how T-duality intertwines generalized complex structures on $N$ and $Y$. In particular we are interested in the mirror symmetry obtained by T-duality on the $T^2$-fibers of $N$. In this case, symplectic structures on $Y$ with respect to which the $T^2$-fibers are Lagrangian are intertwined with complex structures on $Y$ with respect to which the $3$-form is of type $(2,1)+(1,2)$.    

Each of these generalized complex structures gives rise to an action of the $N=2$ super-conformal algebra of central charge $12$ on the zero-winding sectors of the super-symmetric quantum sigma-model with target the underlying manifold. Since T-duality intertwines winding and momentum operators, a full understanding of mirror symmetry for these sigma-models requires adding winding sectors to both sides. In the spirit of DFT, we address this problem by lifting the T-dual super-symmetric sigma-models to an 8-dimensional nilmanifold $\Xe$ where $X$ is the double twisted torus considered in \cite{heluanaldi} and  $T^2$ is a $2$-Torus. Geometrically, both the symplectic structure on $N$ and the complex structure on $Y$ can be realized as complex structures on $\Xe$. We show that the corresponding $N=2$ super-conformal structures also admits lifts to $\Xe$ which are then intertwined by the manifest T-duality.

More generally, complex structures on $\Xe$ which are orthogonal with respect to the tautological pairing on $\brXe$ are in correspondence with generalized complex structures on both $N$ and $Y$ or, equivalently, with $N=2$ super-conformal structures on the corresponding sigma models. 
\end{nolabel}

\begin{nolabel}
The paper is organized as follows. In Section \ref{sec:super-loops} we describe a Poisson structure on the space of super-loops on the double twisted torus and its quantization. In particular, we show that bosonic and fermionic sectors of the quantum theory do not interact. In Section \ref{sec:super-conformal} we add the auxiliary circle directions and explain the relationship between $N=2$ super-conformal structures on the resulting quantum theory on the one hand and complex structures on the $8$-dimensional nilmanifold $\Xe$ on the other. In Section \ref{sec:mirror} we reinterpret these $N=2$ super-conformal structures from the point of view of the T-dual pair of $4$-dimensional targets $(Y,N)$. We also describe global moduli of $N=2$ super-conformal structures both from the $8$-dimensional and $4$-dimensional perspectives.  
\end{nolabel}

{\noindent}{\bf Acknowledgments:} The work of M.A.\ was supported in part by the NSF FRG grant DMS-0854965. R.H. is supported in part by CNPq and Faperj. M.A.\ would like to thank IMPA for hospitality and excellent work conditions while the present work was in the initial stages of preparation. 

\section{Super-loops on the double twisted torus}\label{sec:super-loops}

\begin{nolabel}
Let $X$ be a Poisson manifold and let $Y$ be an arbitrary smooth manifold. The space of smooth maps $\maps(Y,X)$ can be regarded as an infinite-dimensional Poisson manifold in the following way. Let $\{y^i\}$ be local coordinates on $Y$ and consider $\phi = \phi(y^i) \in \maps(Y,X)$. Given local coordinates $\{x^\mu\}$ on $X$, $\phi^\mu = \phi^\mu(y^i) = x^\mu \circ \phi$ may be regarded as local coordinates on $\maps(Y,X)$. The algebra $\mathcal{V}$ generated by $\phi^\mu(y^i)$ and its derivatives $\partial_{y^{i_1}} \dots \partial_{y^{i_k}} \phi^\mu(y^i)$ is equipped with the local bracket \cite{fadeev-tak} such that
\begin{equation} \left\{ \varphi^\mu(y), \phi^\nu(y') \right\} = \left\{ x^\mu, x^\nu \right\}(y') \delta(y - y'), 
\label{eq:bracket1}
\end{equation}
where $\delta(y-y')$ is the functional $\int_Y dy' \delta(y-y') f(y') =  f(y)$, and (\ref{eq:bracket1}) is extended to all of $\mathcal{V}$ by the Leibniz rule. The space of \emph{local functionals} on $\maps(Y,X)$ is the quotient by total derivatives $\cA = \mathcal{V}/(\sum_i \partial_{y^i} \mathcal V)$. The quotient map is denoted by $f \mapsto \int_Y f$ which allows us to write local functionals as $
\int_Y dy f(\phi, \partial \phi, \dots)$. Integrating by parts, the bracket \eqref{eq:bracket1} descends to a Lie bracket on $\cA$. Since the product of local functionals is not in general a local functional, $\cA$ is not quite a Poisson algebra. Instead,  \eqref{eq:bracket1} endows the space $\mathcal{V}$ with the structure of a \emph{Poisson vertex algebra} or \emph{Coisson algebra} \cite{barakat}. 

The prototypical example of this situation occurs when $X = T^* Z$ with its canonical Poisson structure and $Y= S^1$. In this case, (\ref{eq:bracket1}) is equivalent to the standard \emph{symplectic} structure on the loop space $T^*\cL Z = \cL X$. 

This construction works with minor sign modifications in the \emph{super} situation. The main focus of this paper is the case when $Y = S^{1|1}$ is a super-circle and thus $\maps(Y,X)$ can be thought of as the \emph{super-loop space} $\cL^sX$ of $X$. In this section, we detail the construction of the resulting Poisson vertex super-algebra structure in the particular case where $X$ is a certain nilmanifold of real dimension $6$. 
\label{no:loop-spaces-pvas}
\end{nolabel}
\begin{nolabel}\label{no:derivationsuper}
In the notation of \ref{no:loop-spaces-pvas}, let $\sigma$ be a coordinate on $Y=S^1=\mathbb{R}/\mathbb{Z}$ and let $X$ be of dimension $1$ with local coordinate $x$. The vector field $\partial_\sigma$ gives rise to a derivation of $\cV$ according to the formula
\[ 
\partial = \frac{\delta }{\delta \sigma} := \sum_{k \geq 0}  \phi^{(k+1)}  \frac{\partial}{ \partial \phi^{(k)}}\,,
\]
where $\phi^{(k)}(\sigma) = \partial^k_\sigma \phi(\sigma)$. The other case of interest is the super-circle $Y=S^{1|1}$ with even coordinate $\sigma$ and odd coordinate $\zeta$. In this case, $\partial_\sigma$ admits an odd square root $\susy$ corresponding to the odd vector field $\partial_\zeta + \zeta \partial_\sigma$. In other words, $\susy$ is an odd endomorphism of the $\mathbb{Z}/2\mathbb{Z}$-graded algebra $\cV$ such that $\susy^2 = \partial$. The resulting structure on $\cV$ is referred to as an $N_K=1$ SUSY Poisson vertex algebra in \cite{heluani3}. 
\end{nolabel} 
\begin{nolabel}
In the Hamiltonian approach to the sigma-model, one is interested in \emph{quantizing} the super-manifold $\maps(Y,X)$. This amounts to constructing a Hilbert super-space $\mathcal H$ and replacing the \emph{coordinates} $\phi^\mu(y)$ (or more generally the elements of $\mathcal{V}$) with ${\rm End}(\mathcal H)$-valued distributions whose commutators are dictated by the \emph{Poisson brackets} \eqref{eq:bracket1}. For instance, let $Y = S^1$ and $X = S^1 \times S^1$. If $\sigma$ is a coordinate on $Y$ and $(x^0,x_0^*)$ are coordinates on $X$, the general loop in $\maps(Y,X)$ can be expanded using formal Fourier series as
\[ 
\phi(\sigma) = \phi^0(\sigma) = \omega^0 \sigma + \sum_{n \in \mathbb{Z}} x^n e^{- 2 \pi \sqrt{-1} n \sigma}, \qquad \phi^*(\sigma)=  \phi^1(\sigma) = \omega_0^* \sigma +  \sum_{n \in \mathbb{Z}} x^*_n e^{- 2 \pi \sqrt{-1} n \sigma}, 
\]
where $\omega^0, \omega_0^*\in \mathbb{Z}$, $x^n, x_n^* \in \mathbb{R}$, $n \neq 0$ and $x^0, x^*_0$ are well defined modulo $\mathbb Z$. If $X$ is endowed with the symplectic structure such that $\{x^*_0, x^0\} = 1$, \eqref{eq:bracket1} reads 
\[ 
\{x^*_m, x^n\} = \frac{1}{m} \delta_{m,-n} \quad m,n \neq 0, \qquad \left\{ \omega_0^*, x^0 \right\} = \left\{ x^*_0, \omega^0 \right\} = 1\,.
\]  
In this case, or more generally whenever $Y = S^1$ and $\pi_1(X)$ is abelian, vertex algebras provide an adequate algebraic framework for quantization.
\end{nolabel}
\label{sec:basic-constr}
\begin{nolabel}
Let $\Lambda \simeq \mathbb{Z}^3$ be a free $\mathbb{Z}$-module of rank $3$ with basis $\left\{ \omega^*_i \right\}$. We denote by $\varepsilon_{ijk}$ the totally antisymmetric tensor, thought of as an isomorphism $\det: \wedge^2 \Lambda \xrightarrow{\sim} \Lambda^\vee := \Hom(\Lambda, \mathbb{Z})$. 

Let $\left\{ \omega^i \right\}$ be the dual basis of $\Lambda^\vee$. Let $R$ be a super-commutative $\mathbb{R}$-algebra (for simplicity the reader may take $R=\mathbb{R}$), let $V:= \Lambda \otimes_{\mathbb{Z}} R$ and let $V^*:=\Lambda^\vee \otimes_\mathbb{Z} R$ be its dual space. The bases $\left\{\omega^i\right\}$ and $\left\{\omega_i^* \right\}$ define coordinates $\{x^1,x^2,x^3\}$ on $V$ and $\{x_1^*,x_2^*,x_3^*\}$ on $V^*$, respectively. In these coordinates, the $3$-form $\det$ can be written as $\varepsilon_{ijk} dx^i dx^j dx^k$. Consider the Lie group $G$ defined as a central extension: 
\[ 
0 \rightarrow V^* \rightarrow G \rightarrow V \rightarrow 0\,, 
\]
with multiplication given by 
\[ 
(v^*,v) \cdot (w^*, w) = \left( v^* + w^* + \frac{1}{2}\det(v,w,\cdot), v+w \right)\,,
\]
for all $v,w\in V$ and $v^*,w^*\in V^*$.
Quotienting by the co-compact subgroup $\Gamma \subset G$ generated by $\Lambda$ and $\Lambda^\vee$ we obtain a compact $6$-dimensional nilmanifold $X_R = \Gamma \backslash G$. If $R=\mathbb{R}$, we denote this manifold simply by $X$. The projection $G \rightarrow V$ induces a $T^3$-fibration  $\pi:X\rightarrow \Lambda \backslash V \simeq T^3$ which is topologically non-trivial as $\pi^1(X) = \Gamma$ is not abelian. 
\label{no:G-definition}
\end{nolabel}
\begin{nolabel}
Let $\fg$ be the Lie algebra of $G$. Since $\fg$ is a 2-step nilpotent Lie algebra, whose only non-trivial brackets are of the form $[v,w] = \det(v,w,\cdot)$, the exponential map is a diffeomorphism from $\fg$ to $G$. The coordinate expression for $\det$ shows that $\fg$ admits a basis $\left\{ \alpha^i, \beta_i \right\}_{i=1}^3$ with respect to which the only non-vanishing brackets are
\begin{equation} 
\label{eq:fgbrackets} [\beta_i, \beta_j] = \varepsilon_{ijk} \alpha^k. 
\end{equation}
Moreover, $\fg$ is endowed with an ad-{\it invariant} non-degenerate symmetric bilinear pairing $(\cdot\,,\cdot) \in \Sym^2 \fg^*$ and an ad-{\it invariant} non-degenerate alternating pairing $\omega \in \wedge^2 \fg^*$ such that $V$ and $V^*$ are isotropic for both pairings and furthermore
\[ 
(v^*, v) = \omega(v^*, v) = v^*(v), \qquad v\in V, \: v^* \in V^*\,.
\]
The Cartan form $([\cdot,\cdot],\cdot) \in \wedge^3 \fg^*$ is ad-invariant and its restriction to $\wedge^3 V^*$ coincides with $\det$. Since left-invariant vector fields on $G$ provide a trivialization of the tangent bundle of $X$ as $TX \simeq X \times \fg$, these structures can be interpreted geometrically on $X$. Namely, $\omega$ corresponds to a $2$-form $\omega \in \wedge^2 T^*X$ while the Cartan form corresponds to the pullback $\pi^*\!\det$ of the orientation form defined by $\det$ (which we denote by the same symbol) on $\Lambda\backslash V\cong T^3$. These two differential forms on $X$ are related by $d\omega = \pi^*\!\det$. In particular, $\pi^*\!\det$ is exact even though $\det$ is not zero in $H^3(\Lambda \backslash V)$. In the language of \cite{severa}, $(X,\omega)$ is a \emph{twisted Poisson manifold} with background $\pi^*\!\det$.

\label{no:twisted-Poisson-structure}
\end{nolabel}
\begin{nolabel}\label{no:super-nilmanifold}
The super-manifold $T[1]X \simeq X \times \fg[1]$ is a key ingredient in the construction of the super-symmetric theories of this paper. The super-commutative algebra of functions on $T[1]X$ is a free module of rank $6$ over $C^\infty(X)$. The fermionic counterparts $\left\{ \psi^i, \psi^*_i \right\}_{i=1}^3$ of the basis vectors $\left\{ \alpha^i, \beta_i \right\}_{i=1}^3$ are odd generators of the $C^\infty(X)$-module $C^\infty(X)\otimes \fg[1]$. Locally, the super-manifold $T[1]X$ has $6$ even coordinates $x^i,x^*_i$ and $6$ odd coordinates $\psi^i, \psi^*_i$. 

$T[1]X$ can also be realized as a global quotient as follows. The Lie algebra $\fg$ admits a super-extension $\fg_\super := \fg \ltimes \fg[1]$, where the semidirect product is taken with respect to the adjoint action and the super-brackets in $\fg[1]$ vanish. Then $T[1]X \simeq \Gamma \backslash G_\super$, where  $G_\super \simeq G \ltimes \fg[1]$ is the Lie super-group integrating $\fg_\super := \fg \ltimes \fg[1]$. In this sense, $T[1]X$ can be thought of as a ``super-nilmanifold''. 
\end{nolabel}

\begin{nolabel}\label{no:super-g}
Let $\fg$ be the Lie algebra of \ref{no:twisted-Poisson-structure} with its symmetric invariant pairing $(\cdot\,,\cdot)$. Consider the affine Kac-Moody Lie algebra $\hat{\fg} = \fg( (t)) \oplus \mathbb{C} K$. It is a central extension of the complexification of the Lie algebra $\maps(S^1, \fg)$ of loops in $\fg$. We also have the affine Kac-Moody vertex operator algebra $V(\fg)$ is generated by fields $\left\{ \alpha^i,\beta_i \right\}$, $i=1,2,3$ whose non-vanishing OPEs are
\begin{equation} \label{eq:affine1}
\beta_i(z) \cdot \beta_j(w) \sim \frac{\varepsilon_{ijk}\alpha^k(w)}{z-w}, \qquad \beta_i(z) \cdot \alpha^j(w) \sim \frac{\delta_i^j}{(z-w)^2}. 
\end{equation}
Passing from $S^1$ to the super-circle $S^{1|1}$, we obtain a Lie super-algebra $\maps(S^{1|1} , \fg)$ of super-loops into $\fg$, its central extension $\hat{\fg}_\super$ and the vertex super-algebra $V(\fg_\super)$ \cite{kac:vertex} generated by the even fields $\left\{ \alpha^i,\beta_i \right\}_{i=1}^3$ and by the odd fields $\left\{ \bar{\alpha}^i, \bar{\beta}_i \right\}_{i=1}^3$ whose  only non-trivial OPEs are given by \eqref{eq:affine1} as well as
\begin{equation}\label{eq:affine2}
\beta_i(z) \cdot \bar{\beta}_j(w) \sim \frac{\varepsilon_{ijk} \bar{\alpha}^k(w)}{z-w}, \qquad \bar{\beta_i}(z) \cdot \bar{\alpha}^j(w) \sim \frac{\delta_i^j}{z-w}\,.
\end{equation}
In terms of super-fields, $V(\fg_\super)$ is generated by the odd super-fields
\begin{equation}\label{eq:affine3}
\bar{a}^i(z,\theta) := \bar{\alpha}^i(z) + \theta \alpha^i(z), \qquad \bar{b}_i(z,\theta) := \bar{\beta}_i(z) + \theta \beta_i(z)
\end{equation}
and the OPEs \eqref{eq:affine1}-\eqref{eq:affine2} can be rewritten as
\begin{equation}\label{eq:affine4}
\bar{b}_i (z,\theta) \cdot \bar{b}_j(w,\eta) \sim \frac{\varepsilon_{ijk} \bar{a}^k}{z-w-\theta\eta}, \qquad \bar{b}_i (z,\theta) \cdot \bar{a}^j(w,\eta) \sim \frac{\delta_i^j (\theta - \eta)}{(z-w-\theta\eta)^2}\,. 
\end{equation}
\end{nolabel}

\begin{nolabel}
Given any super-commutative $\mathbb{R}$-algebra $R$, consider the affine super-space $\mathbb{A}^{1|1}= \Spec R[\sigma, \zeta]$ where $\sigma$ is even and $\zeta$ is odd. Let $\mathbb{Z}$ act on $\mathbb{A}^{1|1}$ by $(\sigma,\zeta) \mapsto (\sigma + n, (-1)^n \zeta)$ and let $S^{1|1}$ be the quotient $\mathbb{Z} \backslash \mathbb{A}^{1|1}$. A general $R$-point $\varphi$ of the \emph{super-loop space} $\cL^sX := \mathrm{Maps} (S^{1|1}, X)$ can be described by choosing a lifting $\bar \varphi\in \maps(\mathbb R^{1|1},G)$. This defines a map $\bar\varphi^* \in \Hom_R(\cO_G, R[\sigma,\zeta])$, where $\cO_G$ is the algebra of functions in $G$. Notice that a choice of splitting $G \simeq V \times V^*$ defines an isomorphism $\cO_G\cong\Sym^\bullet_R(V\oplus V^*)$. The action of $\Gamma$ on $G$ by left-multiplication induces a representation of $\Gamma$ on $\cO_G$ defined by $(\gamma\cdot f)(g) := f(\gamma^{-1} \cdot g)$ for all $f\in \cO_G$, $\gamma\in \Gamma$ and $g\in G$. Similarly, $\mathbb{Z}$ acts on $R[\sigma,\zeta]$ by $(n\cdot p)(\sigma,\zeta):= p(\sigma - n, (-1)^n \zeta)$. In terms of the pullback homomorphism $\bar\varphi^*$, the equivariance of $\bar\varphi$ is equivalent to the existence, for every $n\in \mathbb Z$ of a $\gamma \in \Gamma$ such that
\begin{equation} \left( n \cdot \bar\varphi^*(f) \right) = \bar\varphi^*(\gamma \cdot f)
\label{eq:equivariance1}
\end{equation}
for all $f\in \cO_G$.
\end{nolabel}

\begin{nolabel}
Since it is enough to check \eqref{eq:equivariance1} on generators of $\cO_G$, we may describe the equivariance condition for $\bar \varphi$ more explicitly as follows. If $\gamma=(\omega_i^*,\omega_i) \in \Gamma$, then 
\[
\gamma \cdot x^i = x^i - \omega^i, \qquad \gamma \cdot  x^*_i = x^*_i - \omega^*_i - \frac{1}{2} \varepsilon_{ijk} \omega^j x^k\,, 
\] 
where we sum over repeated indices. Given a loop $\varphi \in \cL^s X$ we write, with a slight abuse of notation, $x^i(\sigma,\zeta)$ (resp $x^*_i(\sigma,\zeta)$) for the function $\bar\varphi^* (x^i) \in R[\sigma,\zeta]$ (resp. $\bar\varphi^*(x^*_i) \in R[\sigma,\zeta]$). In this way, every super-loop in $\cL^s X$ can be described in terms of its components $x^i(\sigma,\zeta)$, $x^*_i(\sigma,\zeta)$ on which we impose the constraints 
\begin{equation}
x^i(\sigma+1,-\zeta) = x^i(\sigma,\zeta)+ \omega^i, \qquad x^*_i(\sigma+1,-\zeta) = x^*_i(\sigma,\zeta) + \omega^*_i + \frac{1}{2} \varepsilon_{ijk} w^j x^k(\sigma, \zeta)\,, 
\label{eq:equivariance2}
\end{equation}
for some $(\omega^*_i, \omega^i) \in \Gamma$ and for all $i=1,2,3$. 
\label{no:superloops}
\end{nolabel}

\begin{nolabel}
Formally, the general solution of \eqref{eq:equivariance2} is of the form
\begin{equation}
\begin{aligned}
x^i(\sigma,\zeta) &= \omega^i \sigma + \sum_{n \in \mathbb{Z}} x^i_n e^{-2 \pi i n \sigma}  + \zeta \sum_{n \in \tfrac{1}{2} + \mathbb{Z}} \psi^i_n e^{-2 \pi i n \sigma} ,\\
x^*_i(\sigma,\zeta) &=  \omega^*_i \sigma + \sum_{n \in \mathbb{Z}} x^*_{i,n} e^{-2 \pi i n \sigma}  + \zeta \sum_{n \in \tfrac{1}{2} + \mathbb{Z}} \psi^*_{i,n} e^{-2 \pi i n \sigma} + \frac{\sigma}{2} \varepsilon_{ijk} \omega^j x^k(\sigma,\zeta),  \\
\end{aligned}
\label{eq:loop-solutions}
\end{equation}
where $(\omega^*_i, \omega^i) \in \Gamma$, $x^i_n \in V^*$ and $x^*_{i,n} \in V$ are \emph{even}\footnote{Recall that $x^i$ (resp. $x^{*}_i$) is viewed as a coordinate of $V$ (resp. $V^*$).} while $\psi^i_n \in V^*$ and $\psi^*_{i,n} \in V$ are \emph{odd}\footnote{Here we see that we need to work over coordinate rings $R$ with odd parameters.}. Notice that the identity $\varepsilon_{ijk}\omega^i \omega^j x^k = 0$ is needed in order to check that the expansion \eqref{eq:loop-solutions} does indeed satisfy the equivariance constraint \eqref{eq:equivariance2}.
\end{nolabel}

\begin{nolabel}
The use of complex exponentials in (\ref{eq:loop-solutions}), requires us to complexify our variables and work with $(\mathbb C^*)^{1|1}$ instead of $S^{1|1}$. Moreover, we are implicitly stepping outside the algebraic category into the differentiable category. At the cost of introducing logarithms, one can return to the realm of super formal power series by the super-conformal change of coordinates $z = e^{2 \pi \sqrt{-1} \sigma}$, $\theta = ( 2\pi \sqrt{-1})^{1/2}\zeta e^{\pi \sqrt{-1} \sigma}$. In terms of these new coordinates, the fields \eqref{eq:loop-solutions} can be expanded as
\begin{equation}
\begin{aligned}
x^i(z,\theta) &=  \frac{\log(z)}{2 \pi \sqrt{-1}} \omega^i + \sum_{n \in \mathbb{Z}} x^i_n z^{-n}  + \frac{\theta}{(2 \pi \sqrt{-1})^{1/2}} \sum_{n \in \tfrac{1}{2} + \mathbb{Z}} \psi^i_n z^{-n - 1/2} \\
x^*_i(z,\theta) &=  \frac{\log(z)}{2 \pi \sqrt{-1}} \omega^*_i + \sum_{n \in \mathbb{Z}} x^*_{i,n} z^{-n} + \frac{\theta}{(2 \pi \sqrt{-1})^{1/2}} \sum_{n \in \tfrac{1}{2} + \mathbb{Z}} \psi^*_{i,n} z^{-n -1/2} \\ & \quad + \frac{\log(z)}{4\pi \sqrt{-1}} \varepsilon_{ijk} \omega^j x^k(z,\theta)\,.
\end{aligned}
\label{eq:exponentiated-fields}
\end{equation}
Setting $\theta=0$, we recover the expansion of loops in $X$ found in (\cite{heluanaldi}, Section 2.2).
\end{nolabel}

\begin{nolabel}
The presence of logarithms in \eqref{eq:exponentiated-fields} means that the fields $x^i(z,\theta)$, $x_i^*(z,\theta)$ are multi-valued. One way to obtain single-valued fields is to differentiate and look at  currents instead. The trivialization $T^*X\simeq X\times \fg^*$, provided by the the globally defined left-invariant differential forms
\begin{equation} \label{eq:forms-invariant}
dx^i, \qquad dx^*_i - \frac{1}{2} \varepsilon_{ijk} x^j dx^k 
\end{equation}
on $X$ suggests that the right currents to consider are
\begin{equation} \label{eq:currents-well-defined-0}
\susy x^i(z,\theta), \qquad \susy x^*_i (z,\theta) - \frac{1}{2} \varepsilon_{ijk} x^j(z,\theta) \susy x^k(z,\theta)\,,
\end{equation}
where $\susy$ is the odd derivation that acts on our fields as $\partial_\theta + \theta \partial_z$, and extends the vector field $\partial_\zeta+\zeta\partial_\sigma$ of \ref{no:derivationsuper}. A straightforward calculation shows that the currents (\ref{eq:currents-well-defined-0}) can be expressed in terms of the components \eqref{eq:exponentiated-fields} as 

\begin{equation}\label{eq:currents-well-defined}
\begin{gathered}
\susy x^i(z,\theta) = \frac{1}{(2 \pi \sqrt{-1})^{1/2}} \sum_{n \in 1/2 + \mathbb{Z}} \psi^i_n z^{-n -1/2} + \frac{1}{2 \pi \sqrt{-1}} \theta z^{-1} \omega^i - \theta \sum_{n \in \mathbb{Z}} n x^i_n z^{-1-n}, \\
\begin{split}
\susy x^*_i(z,\theta) &- \frac{1}{2} \varepsilon_{ijk} x^j(z,\theta) \susy x^k(z,\theta) = \\ & \frac{1}{(2\pi \sqrt{-1})^{1/2}}\left( \sum_{n \in 1/2+\mathbb{Z}} \psi^*_{i,n} z^{-n-1/2} - \frac{1}{2} \varepsilon_{ijk} \sum_{\stackrel{n \in \mathbb{Z}}{m \in 1/2 + \mathbb{Z}}} x^j_{n} \psi^k_{m} z^{ -1-m-n} \right) +  \\ 
& \theta \left\{  \frac{\omega^*_i}{2 \pi \sqrt{-1}} z^{-1} - \sum_{n \in \mathbb{Z}} x^*_{i,n} z^{-1-n} + \frac{1}{2 \pi \sqrt{-1}} \varepsilon_{ijk} \omega^j \sum_{n \in \mathbb{Z}}x^k_n z^{-1-n} \right. \\& + \left. \frac{1}{2} \varepsilon_{ijk} \sum_{m \in \mathbb{Z}} mx^j_n x^k_m z^{-1-m-n} - \frac{1}{4 \pi \sqrt{-1}} \varepsilon_{ijk} \sum_{m,n \in 1/2 + \mathbb{Z}} \psi^j_m \psi^k_n z^{-n-m-1} \right\}.
\end{split}
\end{gathered}
\end{equation}
Since these expansions contain no logarithms, the super-fields (\ref{eq:currents-well-defined-0}) are well defined.
\end{nolabel}

\begin{nolabel}\label{rescaling}
We would like to use the explicit expansion of the super-loops on $X$ given in \eqref{eq:exponentiated-fields} to \emph{quantize} the infinite-dimensional Poisson super-manifold $\cL^sX$. We impose the canonical quantization relations in two steps. The super-fields \eqref{eq:currents-well-defined} define super-loops in $\cL^s T^*X \simeq \cL^s TX \simeq \cL^s (X \times \fg)$ which are invariant under the left action of $G$ i.e.\ super-loops in $\maps(S^{1|1}, \fg)$. It is therefore natural to identify the super-fields \eqref{eq:currents-well-defined} with the super-fields $\left\{ \bar{a}^i, \bar{b}_i \right\}_{i=1}^3$ that generate $V(\fg_\super)$ as in \eqref{eq:affine3}. After promoting $\left\{ \frac{\omega^i}{2\pi \sqrt{-1}}, \frac{\omega^*_i}{2\pi \sqrt{-1}} \right\}$ to operators $\left\{ W^i,W^*_i \right\}$ and rescaling $\psi^j_n \mapsto (2 \pi \sqrt{-1})^{1/2} \psi^j_n$, $\psi^*_{j,n} \mapsto (2 \pi \sqrt{-1})^{1/2}\psi^*_{j,n}$ for all $n\in \mathbb Z$, we may write
\begin{align}
\susy x^i(z,\theta) &= \bar a^i(z,\theta)=\bar \alpha^i(z) + \theta \alpha^i(z)\,, \\ 
\susy x^*_i(z,\theta) &- \frac{1}{2} \varepsilon_{ijk} x^j(z,\theta) \susy x^k(z,\theta) = \bar b_i(z,\theta)= \bar\beta_i(z) + \theta \beta_i(z)\,,
\end{align}
where
\begin{equation}\label{eq:super-currents-components}
\begin{aligned}
\bar{\alpha}^i(z) &= \sum_{n \in 1/2 + \mathbb{Z}} \psi^i_n z^{-n -1/2}, \qquad \alpha^i(z)=  z^{-1} W^i - \sum_{n \in \mathbb{Z}} n x^i_n z^{-1-n}, \\
\bar{\beta}_i(z) &=  \sum_{n \in 1/2+\mathbb{Z}} \psi^*_{i,n} z^{-n-1/2} - \frac{1}{2} \varepsilon_{ijk} \sum_{\stackrel{n \in \mathbb{Z}}{m \in 1/2 + \mathbb{Z}}} x^j_{n} \psi^k_{m} z^{ -1-m-n},\\
\beta_i(z) &= 
  W^*_i z^{-1} - \sum_{n \in \mathbb{Z}}n x^*_{i,n} z^{-1-n} + \varepsilon_{ijk} W^j \sum_{n \in \mathbb{Z}}x^k_n z^{-1-n} + \\ &\quad   \frac{1}{2} \varepsilon_{ijk} \sum_{m,n \in \mathbb{Z}} mx^j_n x^k_m z^{-1-m-n} - \frac{1}{2} \varepsilon_{ijk} \sum_{m,n \in 1/2 + \mathbb{Z}} \psi^j_m \psi^k_n z^{-n-m-1}. 
\end{aligned}
\end{equation}
Imposing the OPEs \eqref{eq:affine4}, we obtain that in particular $[W_i^*,W_j^*]=-\varepsilon_{ijk} W^k$. This suggests the identification
\begin{equation}\label{eq:Waction}
W^*_i = \partial_{x^i} + \frac{1}{2} \varepsilon_{ijk} x^j \partial_{x^*_k}, \qquad W^i = \partial_{x^*_i}\,,
\end{equation}
and thus
\begin{equation}\label{eq:zero-modes-commutations}
\begin{aligned}
{[}W^*_i, x^j_0] &= [W^j, x^*_{i,0}]  = \delta_i^j \hbar, \\
[W^*_i, x^*_{j,0}] &= - \frac{1}{2} \varepsilon_{ijk} x^k_{0}. 
\end{aligned}
\end{equation}
Combining \eqref{eq:zero-modes-commutations} with the commutators deduced from \eqref{eq:affine4} we arrive at the full set of canonical commutators
\begin{equation}\label{eq:lie-super-algebra}
\begin{aligned}
{[}\psi^*_{i,m}, \psi^j_n] &= n [x^*_{i,m}, x^j_n] = \delta_i^j \delta_{m,-n} \hbar, \\ 
[W^*_i, x^j_n] &= [W_j, x^i_n] =  \delta_i^j \delta_{n,0}\hbar , \\ 
[W^*_i, \psi^*_{j,n}] &= -\frac{1}{2} \varepsilon_{ijk} \psi^k_n, \\
[x^*_{i,m}, \psi^*_{j,n}] &= \frac{1}{2m} \varepsilon_{ijk} \psi^k_{m+n}, \qquad &m \neq 0, \\
[W^*_i, W^*_j] &= - \varepsilon_{ijk} W^k, \\ 
[W^*_i, x^*_{j,n}] &= -\frac{1}{2} \varepsilon_{ijk} x^k_n, \\ 
[x^*_{i,m}, x^*_{j,n}] &= \frac{1}{2} \varepsilon_{ijk} \frac{m+n}{mn} x^k_{m+n} + \frac{\delta_{m,-n}}{m^2} \varepsilon_{ijk} W^k, \qquad &m,n \neq 0 \\ 
[x^*_{i,n}, x^*_{j,0}] &= \frac{1}{2n} \varepsilon_{ijk} x^k_n, \qquad &n \neq 0.
\end{aligned}
\end{equation}
which encodes the Poisson structure on $\cL^sX$. 
\end{nolabel}
\begin{rem}
The denominator $m^2$ in the bracket  $[x^{*}_{i,m}, x^*_{j,n}]$ is responsible for the appearance of dilogarithmic singularities in the OPEs of the corresponding vertex operators in \cite{heluanaldi}.
\label{rem:dilog1}
\end{rem}

\begin{nolabel}
Let $\cL_0$ be the $13$-dimensional Lie algebra spanned by $W^i$, $W_i^*$, their conjugate coordinates $x_0^i$, $x_{i,0}^*$ and by $\hbar$. Let $\cL^b$ the $\mathbb Z$-graded vector space whose degree-$0$ component is $\cL_0$ and whose component of degree $n$ is generated by $x^i_n, x^*_{i,n}$ for all $n \in \mathbb{Z}\setminus\{0\}$. As proved in (\cite{heluanaldi}, Section 2.2), $\cL^b$ is an infinite-dimensional Lie algebra with respect to the commutators \eqref{eq:lie-super-algebra}. This can be extended as follows. Let $\cL^f$ be the $(\frac{1}{2}+\mathbb Z)$-graded vector space generated by $\psi^i_n, \psi^*_{i,n}$, in degree $n \in \frac{1}{2} + \mathbb{Z}$ and let $\cL=\cL^b\oplus \cL^f[1]$. It is straightforward to check the following

\begin{lem*}
$\cL$ is a $\tfrac{1}{2}\mathbb{Z}$ graded Lie super-algebra whose generators obey the commutators \eqref{eq:lie-super-algebra}. 
\end{lem*}
\end{nolabel}

\begin{nolabel}\label{no:pairing}
The Lie algebra $\cL$ admits a canonical invariant super-symmetric bilinear form defined by $(\cL_n, \cL_m) = 0$ if $n+m \neq 0$ and 
\begin{equation}\label{eq:pairing}  (W^i, W^*_j) = (\psi^i_n, \psi^*_{j,-n}) = (x^i_m, x^*_{j,-m}) = \delta^i_j, \qquad n \in \frac{1}{2} + \mathbb{Z}, \quad m \in \mathbb{Z}. 
\end{equation}
\end{nolabel}

\begin{nolabel}\label{no:graded-0-part}
Consider the triangular decomposition
\[
\cL = \cL_{-}  \oplus \cL_0 \oplus \cL_+\,,
\]
where 
$\cL_-$ and $\cL_+$ respectively denote the components of negative and positive degree of $\cL$.  The group $G$ acts on $X$ and thus on $L^2(X)$. The differential of this action defines a representation of $\fg$ on $C^\infty(X)$ which can be extended to $\cL_0$ by letting: $\hbar$ act as the identity, $x^*_{i,0}$ and $x^i_0$ act by multiplication by the local coordinates $x^*_i, x^i$ of $X$, $W^*_i$ and $W_i$ act according to \eqref{eq:Waction}
We further extend this representation to all of $\cL_{\geq 0}$ by requiring that $\cL_+ \cdot C^\infty (X) = 0$. The induced representation
\begin{equation} \label{eq:H-definition}  \cH := \mathrm{Ind}_{\cL_{\geq 0}}^{\cL} C^\infty (X) 
\end{equation}
is compatible with the pairing \eqref{eq:pairing} and can be completed to a Hilbert super-space which we also denote by $\cH$. Setting $\cL^b_+=\cL_+\cap \cL^b$, we recover the Hilbert space of \cite{heluanaldi} as the completion of
\[
\cH^b:= \mathrm{Ind}_{\cL_{\geq 0}^b}^{\cL^b} C^\infty (X)\,.
\]
\end{nolabel}

\begin{nolabel}\label{no:in-terms-of-super-manifold}
Alternatively, $\cH$ can be defined in terms of the action of $G_\super$ on the super-manifold $T[1]X$ described in \ref{no:super-nilmanifold}. In fact, the differential of this action defines an action of the subalgebra $\cL_{-1/2} \oplus \cL_0 \oplus \cL_{1/2}$ on $C^\infty (T[1]X)$ with respect to which:
\begin{enumerate}
\item $\hbar$ acts as the identity;
\item $x^*_{i,0}$, $x^i_0$, $\psi^*_{i,-1/2}$, $\psi^i_{-1/2}$ act by multiplication by the respective local coordinates $x^*_i$, $x^i$, $\psi^*_i$ $\psi_i$;
\item the operators $W^*_i$, $W^i$, $\psi^*_{i,1/2}$, $\psi^i_{1/2}$ act by the derivations 

\begin{equation}
\begin{aligned} 
W^*_i &= \partial_{x_i} + \frac{1}{2} \varepsilon_{ijk} x^j \partial_{x^*_k} + \frac{1}{2} \varepsilon_{ijk} \psi^j \partial_{\psi^*_k}, &&&  W^i &= \partial_{x^*_i}, \\
\psi^*_{i,1/2} &= \partial_{\psi^i},  &&& \psi^i_{1/2} &= \partial_{\psi^*_i}. 
\end{aligned}
\end{equation}
\end{enumerate}
This action is extended to all of the subalgebra $\cL_{\geq -1/2}\subseteq \cL$ by requiring  that $\cL_{> 1/2}$ acts by zero. 
The Hilbert super-space $\cH$ can then be defined as the completion of the induced representation
\begin{equation}\label{eq:super-H}
\cH := \mathrm{Ind}_{\cL_{\geq -1/2}}^\cL C^\infty \left( T[1]X \right). 
\end{equation}
\end{nolabel}
\begin{nolabel}\label{no:in-terms-of-affine}
$\cH$ has a natural structure of module for the super-vertex algebra $V(\fg_\super)$ of \ref{no:super-g}. Indeed the action of $\fg$ on $C^\infty(X)$ (or equivalently the action of $\fg_\super$ on $C^\infty(T[1]X)$) induces a $V(\fg_\super)$ module $\cH'$ defined as follows. We expand the fields  in \ref{no:super-g} as 

\begin{equation}\label{eq:expansion-super-g}
\begin{aligned}
\beta_{i}(z) &:= \sum_{n \in \mathbb{Z}} \beta_{i,n} z^{-1-n}, &&& \alpha^{i}(z) &:= \sum_{n \in \mathbb{Z}} \alpha^i_n z^{-1-n}, \\
\bar{\beta}_i(z) &:= \sum_{n \in \frac{1}{2} + \mathbb{Z}} \bar{\beta}_{i,n} z^{-1/2-n}, &&& \bar{\alpha}^i(z) &:= \sum_{n \in \frac{1}{2} + \mathbb{Z}} \bar{\alpha}^i_n z^{-1/2-n}\,.
\end{aligned}
\end{equation}
The only non-trivial commutators implied by the OPEs \eqref{eq:affine1}-\eqref{eq:affine2} are 
\begin{equation}\label{eq:affine-final}
\begin{aligned}
{[}\beta_{i,m}, \beta_{j,n} ] &=  \varepsilon_{ijk} \alpha^k_{m+n}, &&& [\beta_{i,m}, \bar{\beta}_{j,n} ] &= \varepsilon_{ijk} \bar{\alpha}^k_{m+n} \\ 
[\beta_{i,m}, \alpha^{j}_n] &= m \delta_i^j \delta_{m,-n} \mathbbm{1}, &&& [\bar{\beta}_{i,m}, \bar{\alpha}^j_n] &= \delta_i^j \delta_{m,-n} \mathbbm{1}\,, 
\end{aligned}
\end{equation}
where we denoted the identity of $V(\fg_\super)$ by  $\mathbbm{1}$. These are the commutators of a $\tfrac{1}{2} + \mathbb{Z}$-graded Lie super-algebra $\tilde{\fg}_\super$ generated by
\begin{enumerate}
\item a central element $\mathbbm{1}$ in degree $0$;
\item even vectors $\beta_{i,n}$, $\alpha^i,n$ in degree $n \in \mathbb{Z}$;
\item odd vectors $\bar{\beta}_{i,n}$, $\bar{\alpha}^i_n$ of degree $n \in \tfrac{1}{2} + \mathbb{Z}$\,.
\end{enumerate}
In particular, the degree-$0$ part of $\tilde{\fg}_\super$ is $\fg_\super \oplus \mathbbm{1}$. In terms of the expansion \eqref{eq:expansion-super-g}, the identification \eqref{eq:super-currents-components} between the super-fields \eqref{eq:currents-well-defined-0} and the generating super-fields \eqref{eq:affine3} of $V(\fg_\super)$ becomes 
\begin{equation}\label{eq:identifying-modes}
\begin{gathered}
\bar{\alpha}^i_n = \psi^i_n, \qquad  \alpha^i_0 = W^i, \qquad \alpha^i_n = -n x^i_n \quad (n \neq 0)\\
\bar{\beta}_{i,n} = \psi^*_{i,n} - \frac{1}{2} \varepsilon_{ijk} \sum_{m \in \mathbb{Z}} x^j_{m} \psi^k_{n-m}, \\
\beta_{i,n} = \delta_{i,0} W^*_{i} + \varepsilon_{ijk} W^j x^k_n + \frac{1}{2} \varepsilon_{ijk} \sum_{m \in \mathbb{Z}} m x^j_{n-m} x^k_{m} - \frac{1}{2} \varepsilon_{ijk} \sum_{m \in \frac{1}{2} + \mathbb{Z}} \psi^j_m \psi^k_{n-m}\,.
\end{gathered}
\end{equation}
\end{nolabel}

\begin{nolabel}
The Lie super-algebra $\tilde{\fg}_\super$ admits a non-degenerate symmetric invariant form such that $(\tilde{\fg}_{\super, n}, \tilde{\fg}_{\super,m}) = 0$ if $m+n \neq 0$ and the only non-trivial pairings on the basis elements are:
\begin{equation}\label{eq:invariant-fform-affine}
\left( \bar{\beta}_{i,n}, \bar{\alpha}^j_{-n} \right) = \left(\beta_{i,m}, \alpha^j_{-m}\right) = \delta_i^j, \qquad n \in \frac{1}{2} + \mathbb{Z}, \quad m \in \mathbb{Z}.
\end{equation}
If the representation of $\fg_\super$ on $C^\infty(X)$ is extended to $\tilde{\fg}_{\super, \geq 0}$ by requiring that $\mathbbm{1}$ acts as the identity and that $\tilde{\fg}_{\super, > 0}$ acts as zero, the identification (\ref{eq:identifying-modes}) defines a canonical isomorphism of  $\tilde{\fg}_\super$-modules 
\[
\mathrm{Ind}_{\cL_{\geq 0}}^{\cL} C^\infty (X)\cong \mathrm{Ind}^{\tilde{\fg}_\super}_{\tilde{\fg}_{\super, \geq 0}} C^\infty(X)\,.
\]
Moreover this isomorphism intertwines the symmetric bilinear forms (\ref{eq:pairing}) and  (\ref{eq:invariant-fform-affine}).

\end{nolabel}
\begin{nolabel}\label{no:isomorphism-with-tensor} 
The canonical commutators (\ref{eq:lie-super-algebra}) exhibit a non-trivial pairing between the even and odd components of the super-fields \eqref{eq:exponentiated-fields}. Nevertheless we have the following

\begin{thm*}
The super-Hilbert space $\cH$ factorizes canonically as $\cH^b\otimes \cH^f$ where $\cH^f$ is the Hilbert space of the free fermionic system generated by the fields $\psi^*_i$ and $\psi^i$, $i=1,2,3$ satisfying 
\[ 
\psi^*_i(z) \cdot \psi^j(w) \sim \frac{\delta_i^j}{z-w}\,. 
\]
Moreover, under the state-field correspondence, fields corresponding to vectors in $\cH^b\otimes \vac^f$ have trivial OPEs with fields corresponding to vectors in $\vac^b\otimes\cH^f$.   
\end{thm*}

\begin{proof}
Under the change of variables
\[
\tilde{x}^*_{i,m} := x^*_{i,m} + \frac{1}{4m} \varepsilon_{ijk} \sum_{n \in \frac{1}{2} + \mathbb{Z}} \psi^j_n \psi^k_{m-n}, \qquad \tilde{W}^*_i = W^*_i - \frac{1}{4} \varepsilon_{ijk} \sum_{n \in \frac{1}{2} + \mathbb{Z}} \psi^j_n \psi^k_{-n}
\]
the non-trivial commutation relations become
\begin{equation}\label{eq:lie-super-algebra-tensor}
\begin{aligned}
{[}\psi^*_{i,m}, \psi^j_n] &= n [x^*_{i,m}, x^j_n] = \delta_i^j \delta_{m,-n} \hbar, \\ 
[W^*_i, x^j_n] &= [W_j, x^i_n] =  \delta_i^j \delta_{n,0}\hbar , \\ 
[\tilde{W}^*_i, \tilde{W}^*_j] &= - \varepsilon_{ijk} W^k, \\ 
[\tilde{W}^*_i, x^*_{j,n}] &= -\frac{1}{2} \varepsilon_{ijk} x^k_n, \\ 
[\tilde{x}^*_{i,m}, \tilde{x}^*_{j,n}] &= \frac{1}{2} \varepsilon_{ijk} \frac{m+n}{mn} x^k_{m+n} + \frac{\delta_{m,-n}}{m^2} \varepsilon_{ijk} W^k, \qquad &m,n \neq 0 \\ 
[\tilde{x}^*_{i,n}, \tilde{x}^*_{j,0}] &= \frac{1}{2n} \varepsilon_{ijk} x^k_n, \qquad &n \neq 0\,.
\end{aligned}
\end{equation}
\end{proof}
\end{nolabel}

\section{$N=2$ super-conformal structures}\label{sec:super-conformal}

\begin{nolabel}\label{no:n=2susy}
Recall that the $N=2$ super-vertex algebra of central charge $c$ is generated by a Virasoro field $L$ of central charge $c$, a free boson $J$ primary of conformal weight $1$ and two odd fermions $G^\pm$ primary of conformal weight $3/2$ with remaining non-trivial OPEs
\begin{equation}\label{eq:n=2opes}
\begin{aligned}
J(z) \cdot J(w) &\sim \frac{c}{3} \frac{1}{(z-w)^2}, \qquad J(z) \cdot G^\pm(w) \sim \pm \frac{G^\pm(w)}{z-w}, \\
G^+(z)G^-(w) &\sim \frac{L(w) + \frac{1}{2} \partial_w J(w)}{z-w} + \frac{J(w)}{(z-w)^2} + \frac{c}{12} \frac{1}{(z-w)^3}.
\end{aligned}
\end{equation}
In terms of the super-fields\footnote{We are rescaling here the field $J$ of \cite{heluani3,heluani9} to avoid unnecessary $\sqrt{-1}$ factors.}  
\[ 
\begin{aligned} J(z,\theta) &= J(z) + \theta \left(G^-(z) - G^+(z) \right) \\ H(z,\theta) &= \left( G^+(z) + G^-(z) \right) + 2 \theta L(z)\,,\end{aligned} 
\] 
the OPEs \eqref{eq:n=2opes} read
\begin{equation}\label{eq:n=2susyopes}
\begin{aligned}
H(z,\theta) \cdot H(w,\zeta) &\sim \frac{2 \partial H(w,\zeta)}{z-w-\theta\zeta} + \susy H(w,\zeta)\frac{(\theta - \zeta)}{(z-w-\theta\zeta)^2} +  \frac{3 H(w,\zeta)}{(z-w-\theta\zeta)^2}, \\ 
H(z,\theta) \cdot J(w,\zeta) &\sim \frac{2 \partial J(w,\zeta)}{(z-w-\theta\zeta)} + \frac{J(w,\zeta)}{(z-w-\theta\zeta)^2}, \\ 
J(z,\theta) \cdot J(w,\zeta) &\sim \frac{H(w,\zeta)}{(z-w-\theta\zeta)}  + \frac{c}{3} \frac{\theta - \zeta}{(z-w-\theta\zeta)^3}\,, 
\end{aligned}
\end{equation}
where $\susy = \partial_\zeta + \zeta \partial_w$ is the square root of $\partial=\partial_w$. Notice that the super-field $J$ generates the whole $N=2$ super-conformal algebra in the sense that $H$ is recovered from the OPE of $J$ with itself and $G^\pm$ can be calculated from the even and odd components of $J$ and $H$.
\end{nolabel}

\begin{nolabel}
Let $\fg$ be a simple or abelian finite dimensional Lie algebra with its invariant symmetric bilinear form $(\cdot\,,\cdot)$ and dual Coxeter number $h^\vee$. The vertex super-algebra $V^k(\fg_\super)$ \cite{kac:vertex} (see also \cite[Example 5.9]{heluani3} for the super-field version) is generated by super-fields $\bar{a}$, $a \in \fg$ with OPEs
\[ 
\bar{a}(z,\theta) \cdot \bar{b}(w,\zeta) \sim \frac{\overline{[a,b]}(w,\zeta)}{(z-w-\theta\zeta)} + (k+h^\vee)(a,b) \frac{(\theta-\zeta)}{(z-w-\theta\zeta)^2}, \qquad a,b \in \fg\,. 
\]
In terms of the components of
\[ 
\bar{a}(z,\theta) = \bar{a}(z) + \theta a(z)\quad\textrm{and}\quad \bar{b}(z,\theta)=\bar(b)(z)+\theta b(z) 
\] 
these OPE read
\[ 
\begin{aligned} a(z)\cdot b(w) &\sim \frac{[a,b](w)}{(z-w)} + (k+h^\vee) \frac{(a,b)}{(z-w)^2}, \\ a(z) \cdot \bar{b}(w) &\sim \frac{\overline{[a,b]}(w)}{z-w}, \\ \bar{a}(z) \cdot \bar{b}(w) &\sim \frac{(k+h^\vee) (a,b)}{z-w}\,. 
\end{aligned} 
\]
The following proposition is proved in \cite[Prop. 2.14]{heluani9} as a particular case of a construction of Getzler \cite{getzler}.

\begin{prop}\label{prop:bialgebras}
Let $\fl$ be a Lie bialgebra and $\fg := \fl \oplus \fl^*$ its Drinfeld double. Let $h^\vee$ be the dual Coxeter number of $\fg$. Let $\left\{ e_i \right\}_{i=1}^n$ be a basis of $\fl$ and $\left\{ e^i \right\}_{i=1}^l$ its dual basis of $\fl^*$. The super-field
\[ 
J = \frac{1}{k+h^\vee} \sum_{i=1}^n \bar{e}^i \bar{e}_i \in V(\fg_\super)\,, 
\]
generates a copy of the $N=2$ super-conformal vertex algebra of central charge $c = 3 \dim \fl$  in $V^k(\fg_\super)$. 
\end{prop}
\end{nolabel}

\begin{nolabel}\label{no:zero-coords}
Let $x^0,x_0^*$ be local coordinates on $T^2=\mathbb Z^2\backslash \mathbb R^2$. Super-loops on $T^2$ can be expanded\footnote{a rescaling analogous to the one used in Section \ref{rescaling} being understood.} as
\begin{equation}\label{eq:zero-coords}
\begin{aligned}
x^0(z,\theta) = \log(z) W^0 + \sum_{n \in \mathbb{Z}} x^0_n z^{-n} + \theta \sum_{n \in \frac{1}{2} + \mathbb{Z}} \psi^0_{n} z^{-n-1/2}, \\ 
x^*_0(z,\theta) = \log(z) \omega^*_0 + \sum_{n \in \mathbb{Z}} x^*_{0,n} z^{-n} + \theta\sum_{n \in \frac{1}{2} + \mathbb{Z}} \psi^*_{0,n} z^{-n-1/2}. 
\end{aligned}
\end{equation}
The standard quantization of $\cL^s T^2$ is obtained by requiring the coefficients of these expansions to satisfy the canonical commutation relations
\begin{equation}\label{eq:zero-coords-commut}
[W^0, x^*_0] = [W^*_0, x^0] = \hbar, \qquad [\psi^*_{0,n}, \psi^0_{m}] = n [x^*_{0,m},x^0_n] = \delta_{m,-n} \hbar\,. 
\end{equation}
Denote by $\cL_{T^2}$ the $\frac{1}{2} + \mathbb{Z}$ Lie super-algebra with commutation relations \eqref{eq:zero-coords-commut} and by $\cH_{T^2}$ its module induced from the action of the degree-$0$ part of $\cL_{T_2}$ on $C^\infty (T^2)$. Combining with the $\cL$-module structure of $\cH$ we obtain a $\brcLe$-module structure on (a completion of) $\cHe$ induced by the action of the degree-$0$ part of $\cLe$ on $C^\infty(\Xe)$.

\end{nolabel}
\begin{nolabel}\label{no:tangent-bundle}
Let $\fge$ be the $8$-dimensional Lie algebra obtained as a sum of the $6$-dimensional Lie algebra $\fge$ and a two-dimensional commutative algebra $\mathfrak t^2$. We endow $\mathfrak t^2$ with an invariant pairing of signature $(1,1)$ so that the natural invariant pairing on $\fg$ extends to a natural invariant pairing $(\cdot,\cdot)$ on $\fge$ of signature $(4,4)$. Since $\fge$ is nilpotent, then $h^\vee=0$ and we can set $k=1$ in the construction of $V(\brfge_\super)$ since for any two non-vanishing level the algebras $V^k(\brfge_\super)$ are isomorphic. Here $\brfge_\super=\brfge\ltimes \brfge [1]$ is constructed as in \ref{no:super-nilmanifold}. An argument similar to that of \ref{no:in-terms-of-affine} shows that $\cHe$ has a natural structure of $V(\brfge_\super)$-module.

The $8$-dimensional real nilpotent group $\Ge$ has Lie algebra $\fge$ and can be viewed as an extension of $\Ve$ by $\brVe^*$. $\Ge$ contains the discrete co-compact subgroup $\Gammae$, which is a nontrivial extension of $\Lambdae$ by $\brLambdae^vee$. The tangent bundle of the nilmanifold $\Xe=\brGammae\backslash\Ge$ is trivial and isomorphic to $\brXe \times \brfge$. 

Let $\cJ$ be a $\brGe$-equivariant complex structure on $\Xe$ i.e.\ a right-invariant endomorphism of $T\brXe$. Equivalently, $\cJ$ can be thought of as an endomorphism of $\fge$. Such a $\cJ$ gives rise to a decomposition $\fge \simeq \fl \oplus \overline{\fl}$ where $\fl$ (resp. $\overline{\fl}$) is the $\sqrt{-1}$-eigenspace of $\cJ$ (resp. the $-\sqrt{-1}$ eigenspace). Since $\cJ$ is an integrable almost complex structure, $\fl$, $\overline{\fl}$ are both subalgebras of $\fge$. The natural symmetric invariant pairing $(\cdot\,,\cdot)$ of $\fge$ gives rise to a metric of signature $(4,4)$ on $T\brXe$ which can be extended to the complexification $(T\brXe)\otimes \mathbb C$. From now on we restrict our attention to equivariant complex structures $\cJ$ that are $(\cdot\,,\cdot)$-orthogonal. It follows that each $\fl$ and $\overline{\fl}$ is isotropic with respect to $(\cdot\,,\cdot)$ and the triple $(\fge, \fl, \overline{\fl}$) is a Manin triple. Proposition \ref{prop:bialgebras} then gives an $N=2$ structure of central charge $c=12$ on $V(\brfge_\super)$ which therefore acts on $\cHe$. Summarizing, we have
\begin{thm*}
For each $\brGe$-equivariant $(\cdot\,,\cdot)$-orthogonal complex structure $\cJ$ on $\Xe$ there exists a corresponding $N=2$ super-conformal structure on $\cHe$ of central charge $c=12$. 
\end{thm*}
\end{nolabel}

\section{Mirror symmetry}\label{sec:mirror}

In this section we realize the Hilbert super-space $\cHe$ together with the $N=2$ super-conformal structures described in \ref{no:tangent-bundle} in terms of super-symmetric sigma-models with $4$-dimensional targets. In particular we show how the $N=2$ super-conformal structure attached to the Kodaira-Thurston nilmanifold is mirror dual to the $N=2$ super conformal structure of a complex torus twisted by a holomorphic gerbe.

\begin{nolabel}\label{no:kummer}
The manifold $\Xe$ is a $T^4$-fibration over the torus $Y:=\brLambdae \backslash \brVe$. The $3$-form $\det$ on $Y$ represents a non-trivial class in $H^3(Y, \mathbb{Z})$ as in \ref{no:twisted-Poisson-structure}. We consider the complexified standard Courant algebroid on $Y$ twisted by $\det$ i.e.\ $E=TY \oplus T^*Y$ with the \emph{twisted Dorfman bracket} defined by by 
\begin{equation}
\label{eq:dorfman-bracket}
[ \tau + \zeta, \tau'+\zeta'] = Lie_\tau (\tau'+\zeta') - d (\tau', \zeta) + \det(\tau,\tau', \cdot)\,, 
\end{equation}
for all $\tau,\tau' \in TY$ and $\zeta,\zeta' \in T^*Y$, where $Lie_\tau$ denotes the usual Lie derivative with respect to $\tau$. Using the coordinates $\left\{ x^i \right\}_{i=0}^3$ on $Y$, $TY$ and $T^*Y$ are trivialized by the global sections $\partial_{x^i}$ and $dx^i$. The only non-trivial trivial twisted Dorfman brackets between these sections are
\[ 
[\partial_{x^i}, \partial_{x^j} ] = \varepsilon_{ijk} dx^k, \qquad i,j,k=0,\ldots,3\,, 
\]
where $\varepsilon_{ijk} dx^i dx^j dx^k$ is a local expression for $\det$ and the totally antisymmetric tensor is extended by setting $\varepsilon_{ijk} =0$ if at least one of the indices is zero. In particular, setting $\beta_i = \partial_{x^i}$, $\alpha^i = dx^i$ yields a trivialization $E \simeq Y \times \brfge$.

Consider the $N=1$ super-symmetric sigma-model with target $(Y,\det)$ and its Hilbert super-space of states $\cH_Y$, thought of as a quantization of the algebra of local functions on the infinite-dimensional manifold $\cL^s Y$. The connected components of $\cL^s Y$ are labeled by  $\Lambdae = \pi^1(Y)$ so that, similarly to \cite{heluanaldi}, one obtains the decomposition 
\[
\cH_Y = \bigoplus_{\omega \in \Lambdae} \cH_{Y,\omega}\,. 
\]
Vectors in $v \in \cH_{Y, \omega}$ are referred to as having \emph{winding} $\omega$.

The correspondence between the global frame $\{\alpha^i,\beta_i\}$ of $E$ and the currents generating $V(\brfge_\super)$ as in \ref{no:super-g}, make $\cH_Y$ into 
a $V(\brfge_\super)$-module structure. In absence of the gerbe, each summand $\cH_{Y, \omega}$ is the $V(\brfge_\super)$-module generated by the zero modes which are identified with the Fourier basis of $L^2(Y)$ (see for example \cite{kapustin}). As explained in \cite[Appendix]{heluanaldi}, the flux defined by $\det$ can be accounted for as follows. For each $\omega \in \Lambdae \simeq H_1(Y)$, let $\cL_\omega$ be the line bundle on $Y$ with first Chern class $\iota_\omega \det \in H^2(Y)$. If $L^2(\cL_\omega)$ denotes the space of $L^2$-sections of $\cL_\omega$, the space of $0$-modes of $\cH_Y$ is canonically identified with $\bigoplus_{\omega \in \Lambdae} L^2(\cL_\omega)$. We have the following

\begin{prop*}[{\cite[Appendix]{heluanaldi}}]
There is a canonical isomorphism \[ L^2 (\Xe) \simeq \bigoplus_{\omega \in \Lambda} L^2(\cL_\omega). \]
\end{prop*}

\noindent
Through this isomorphism, the natural $\brGe$-action on $L^2(\Xe)$ defines a $\brGe$-action on $\bigoplus_{\omega \in \Lambdae} L^2(\cL_\omega)$. Inducing the differential of this action endows $\cH_Y$ with the structures of $V(\brfge_\super)$-module. By construction we obtain a canonical identification $\cHe \simeq \cH_Y$. 
\end{nolabel}

\begin{nolabel}\label{no:complex} The $4$-dimensional abelian Lie group $\Ve$ acts on $Y = \brLambdae \backslash \brVe$ by right translations. Let $\cJ$ be an equivariant generalized complex structure.  We obtain a decomposition $\fge = \fl \oplus \overline{\fl}$ with each summand being an isotropic subalgebra. By Proposition \ref{prop:bialgebras} this endows $\cH_Y$ with an $N=2$ super-conformal structure and we obtain the following
\begin{prop*} \hfill
\begin{enumerate}
\item For each equivariant generalized Calabi-Yau structure on $Y$ there exists an associated $N=2$ super-symmetric structure of central charge $12$ on $\cH_Y$. 
\item 
There exists a correspondence between equivariant generalized Calabi-Yau structures on $Y$ and equivariant complex structures on $\Xe$ compatible with the tautological pairing on $T\brXe$. 
\item The natural isomorphism $\cHe \simeq \cH_Y$ of \ref{no:kummer} intertwines the $N=2$ super-conformal structures under the correspondence in b).
\end{enumerate}
\end{prop*}
\begin{proof}
Recall that $T\brXe$ and the standard Courant algebroid of $(Y,\det)$ are both trivial with fibers isomorphic to $\fge$ i.e.\ $T\brXe \simeq \brXe \times \brfge$ and $(T\oplus T^*)Y \simeq Y\times \brfge$. Therefore, equivariant complex structures on $\Xe$ and equivariant generalized complex structures on $Y$ both correspond to splittings of $\fge$ by isotropic subalgebras $\fge = \fl \oplus \overline{\fl}$. The proposition follows then follows from the observation \cite{gualtieri-cavalcanti-nilmanifolds} that, since $\fl$ is a nilpotent algebra, any such splitting will have trivial generalized canonical bundle and thus any equivariant generalized complex structure on $Y$ is generalized Calabi-Yau. 
\end{proof}
\end{nolabel}

\begin{rem}
$\cH_Y$ is in fact a module over the chiral de Rham complex \cite{malikov} of $Y$, viewed as the $0$-winding sector of $\cH_Y$ or equivalently as the $V(\brfge_\super)$-module induced from $L^2(Y)$. The $N=2$ structure of $\cH_Y$  corresponding to the generalized Calabi-Yau structure of $Y$ is the one constructed in \cite{heluani9}. 
\end{rem}

\begin{ex}
Consider holomorphic coordinates $z^1 = x^0 + i x^1$, $z^2 = x^2+ i x^3$ on $Y$, where $\{x^i\}_{i=0}^3$ are real coordinates on each $S^1$ factor. With respect to this complex structure
\[ 
\det = \frac{1}{4} \left( dz^1 dz^2\overline{dz^2} - \overline{dz^1} dz^2 \overline{dz^2} \right)\,. 
\]
The $\sqrt{-1}$-eigenbundle $\fl$ of this generalized complex structure is generated by $\partial_{\overline{z^1}}, \partial_{\overline{z^2}}, dz^1, dz^2$ while $\overline{\fl}$ has a framing given by $\partial_{z^1}, \partial_{z^2}, \overline{dz^1}, \overline{dz^2}$. As Lie algebras $\fl = \mathcal{heis} \oplus \mathbb{C}$ is a sum of the $3$-dimensional Heisenberg Lie algebra spanned by $\partial_{\overline{z^1}}, \partial_{\overline{z^2}}$, $d z^2$ and the commutative Lie algebra $\mathbb{C}$ spanned by $dz^1$. The only non-trivial bracket in $\fl$ is given by $[\partial_{\overline{z^1}}, \partial_{\overline{z^2}}] =  dz^2$. 

Recall that the trivialization $(T \oplus T^*)Y \simeq Y \times \brfge$ is given by the identification $\partial_{x^i} = \beta_i$ and $dx^i = \alpha^i$. If we denote
\[ 
\begin{aligned}
\beta_{z^1} &= \beta_0 - i \beta_1, &&& \beta_{z^2} &= \beta_2 - i \beta_3, &&& \alpha^{z^1} &= \alpha^0 + i \alpha^{1}, &&& \alpha^{z^2} &= \alpha^1 + i \alpha^2, \\ 
\beta_{\overline{z^1}} &= \beta_0 + i \beta_1, &&& \beta_{\overline{z^2}} &= \beta_2 + i \beta_3, &&& \alpha^{\overline{z^1}} &= \alpha^0 - i \alpha^{1}, &&& \alpha^{\overline{z^2}} &= \alpha^1 - i \alpha^2\,, 
\end{aligned}
\]
the current $J$ generating the $N=2$ super-conformal structure given by Proposition \ref{prop:bialgebras} becomes
\[ 
J = \frac{1}{2} \left( \beta_{\overline{z^1}} \alpha^{\overline{z^1}} + \beta_{\overline{z^2}} \alpha^{\overline{z^2}} + \alpha^{z^1} \beta_{z^1} + \alpha^{z^2} \beta_{z^2} \right) = -\sqrt{-1} \left( \beta_0 \alpha^1 - \beta_1 \alpha^0 + \beta_2 \alpha^3 - \beta_3 \alpha^2 \right)\,. 
\] 
In terms of the scalar super-fields $x^*_i(z,\theta)$, $x^i(z,\theta)$ and their well defined currents \eqref{eq:currents-well-defined} we see that the $N=2$ super-conformal structure is generated by the super-field
\begin{multline}\label{eq:J1}
J = -\sqrt{-1} \Bigl( \susy x^*_0 \susy x^1 - \susy x^*_1 \susy x^0 + \susy x^*_2 \susy x^3 - \susy x^*_3 \susy x^2 -  \\  \frac{1}{2} \left( x^2 \susy x^0 \susy x^3 + x^3 \susy x^2 \susy x^0 + x^3 \susy x^1 \susy x^3 + x^2 \susy x^3 \susy x^2 \right)   \Bigr)
\end{multline}
\label{ex:complex}
\end{ex}
\begin{nolabel}\label{no:kodaira}
Consider the real $4$-dimensional Lie algebra $\mathbb{C} \oplus \mathcal{heis}$ where $\mathcal{heis}$ is the $3$-dimensional Heisenberg Lie algebra. Exponentiation yields the $4$-dimensional real Lie group $\mathbb{R} \times \mathcal{H}eis$ and its lattice $\mathbb{Z} \times \mathcal{H}eis(\mathbb{Z})$. Here $\mathcal{H}eis(\mathbb{Z})$ can be viewed as the group of upper triangular $3\times 3$ matrices with integer entries with its natural embedding in the real $3$-dimensional group $\mathcal{H}eis$. We have the corresponding $4$-dimensional nilmanifold $N := (\mathbb{Z} \times \mathcal{H}eis(\mathbb{Z})) \backslash (\mathbb{R} \times \mathcal{H}eis)$. It is a non-trivial $S^1$ fibration\footnote{The fiber is given by the copy of $\mathbb{R}$ given by the center of $\mathcal{H}eis$ modulo the copy of $\mathbb{Z}$ given by  center of $\mathcal{H}eis(\mathbb{Z})$} over $T^3$. 

Let us choose coordinates $\left\{ y^i \right\}_{i=0}^3$ on $N$ as follows. $y^0$ is a coordinate on the copy of $\mathbb{R} \subset \mathbb{R} \times \mathcal{H}eis$. And we identify $\mathcal{H}eis$ with $\mathbb{R}^3$ with coordinates $\left\{ y^i \right\}_{i=1}^3$ and multiplication given by
\[ \left( y^1, y^2, y^3 \right) \cdot \left( \tilde{y}^1,\tilde{y}^2, \tilde{y}^3 \right) = \left( y^1 + \tilde{y}^1, y^2 + \tilde{y}^2 + \frac{1}{2} y^1 \tilde{y}^3 - \frac{1}{2} y^3 \tilde{y}^1, y^3 + \tilde{y}^3 \right). \]

The standard Courant algebroid of $N$ is trivialized by the left-invariant vector fields and differential forms on $\mathbb{R} \times \mathcal{H}eis$. If the standard Courant algebroid of $N$ is identified with $N\times \brfge$ by setting
\begin{equation} \label{eq:nilmanifold-framing}
\begin{aligned}
\beta_0 &= \partial_{y^0}, &&& \beta_1 &= \partial_{y^1} - \frac{1}{2} y^3 \partial_{y^2}, &&& -\alpha^2 &= \partial_{y^2}, &&& \beta_3&= \partial_{y^3} + \frac{1}{2} y^1 \partial_{y^2} \\ 
\alpha^0 &= dy^0, &&& \alpha^1 &=dy^1, &&& -\beta_2 &= dy^2 - \frac{1}{2} y^1 dy^3 + \frac{1}{2} y^3 dy^1, &&& \alpha^3 &=dy^3\,,
\end{aligned}
\end{equation}
the standard Dorfman bracket defined by
\[ 
[\tau + \zeta, \tau'+\zeta'] = Lie_\tau (\tau'+\zeta') - d (\tau', \zeta)\,,
\]
for all $\tau,\tau' \in TN$ and $\zeta,\zeta' \in T^*N$ matches the commutators \eqref{eq:fgbrackets} of $\fge$. Let $\cH_N$ be the Hilbert super-space obtained from quantizing the Poisson super-manifold $T^* \cL^s N \simeq T^*[1] T^* \cL N$. The super-currents corresponding to the invariant sections \eqref{eq:nilmanifold-framing} of $(T\oplus T^*)N$ provide $\cH_N$ with the structure of a $V(\brfge_\super)$-module as follows.
The connected components of $\cL^sN$ are parametrized by conjugacy classes in $\mathcal{H}eis(\mathbb{Z})=\pi_1(N)$. For each such conjugacy class $\omega \in \pi^1(N)^\sim$, the space of $0$-modes is a copy of $L^2(N)$.
There is an embedding of $\mathbb{R} \times \mathcal{H}eis$ into the $8$-dimensional nilpotent group $\Ge$, which in terms of the coordinates $\left\{ y^i \right\}$ of $\mathbb{R} \times \mathcal{H}eis$ and $\left\{ x^i, x^*_i \right\}$ of $\Ge$ can be written as 
\[ 
(y^0, y^1, y^2, y^3) \mapsto \left( x^0, x^1, x^*_2, x^3 \right)\, . 
\]
As explained in \cite[Appendix]{heluanaldi} this embedding induces a natural identification \[ L^2(\Xe) \simeq \bigoplus_{\omega \in \pi^1(N)^\sim} L^2(N)\] intertwining the action of $\mathbb{R} \times \mathcal{H}eis$. The Hilbert super-space of states $\cH_N$ is then constructed as the representation of $V(\brfge_\super)$ induced from $\bigoplus_{\omega \in \pi^1(N)^\sim} L^2(Y)$. Hence, by construction, we have an identification $\cHe \simeq \cH_Y$.  
\end{nolabel}

\begin{nolabel}\label{no:calabi-yau-nilmanifold}
The $4$-dimensional group $\mathbb{R} \times \mathcal{H}eis$ acts on $N$ on the right. As in \ref{no:complex} an equivariant generalized complex structure $\cJ$ on $N$ produces a Manin triple $(\fge, \fl, \overline{\fl})$ and therefore an $N=2$ super-conformal structure of central charge $12$ on $\cH_N$. We obtain 

\begin{thm*}\hfill
\begin{enumerate}
\item There is a natural unitary identification $\cH_Y \simeq \cH_N$. 
\item There is a one-to-one correspondence between $\mathbb{R}^4$-equivariant generalized Calabi-Yau structures on $Y$ and $\mathbb{R}\times \mathcal{H}eis$-equivariant generalized Calabi-Yau structures on $N$. 
\item The isomorphism of a) intertwines the $N=2$ super-conformal structures given by the corresponding generalized Calabi-Yau structures of b) and Prop. \ref{prop:bialgebras}. 
\end{enumerate}
\end{thm*}
\end{nolabel}

\begin{ex}
The manifold $N$ admits the equivariant symplectic form
\[
dy^0dy^1 + \left( dy^2 + \frac{1}{2} y^3 dy^1 \right) dy^3\,. 
\]
The graph $\fl$ of the corresponding generalized complex structure on $(T \oplus T^*) N \simeq N \times \brfge$ is isomorphic to $TN$ and is spanned by $\partial_{y^0} -\sqrt{-1} dy^1$, $\partial_{y^2} -\sqrt{-1} dy^3$ as well as 
\[ 
\partial_{y^1} -\frac{1}{2} y^3 \partial_{y^2} +\sqrt{-1} dy^0\,, \qquad \partial_{y^3} + \frac{1}{2} y^1 \partial_{y^2} +\sqrt{-1}\left( dy^2 + \frac{1}{2} y^3 dy^1 - \frac{1}{2} y^1 dy^3\right)\, . 
\] 
The corresponding $N=2$ super-conformal structure is identified with the one in $\cH_Y$ given in Example \ref{ex:complex}.
\end{ex}

\begin{nolabel} We would like to describe moduli of equivariant generalized complex structures on $(Y,\det)$ and $N$ from the point of view of equivariant $(\cdot\,,\cdot)$-orthogonal complex structures on $\Xe$.
Equivariant generalized complex structures on a nilmanifold are in bijection \cite{gualtieri-cavalcanti-nilmanifolds} with left-invariant pure spinors lines $\mathbb C^*\rho$ such that $d\rho=0$ and $\rho\wedge \overline \rho>0$. The same holds in presence of a closed left-invariant 3-form $H$, provided that $d$ is replaced by the twisted de Rham differential $d_H=d+H\wedge$. Consider first the component of the moduli space of equivariant generalized complex structures on $(Y,\det)$ represented by pure spinors of even degree. Such a pure spinor $\rho$ is a $\mathbb C$-linear combination of invariant $d_{\det}$-closed forms of even degree on $Y$ i.e.\ 
\begin{multline*}
\rho=a_{01}dx^0dx^1+a_{02}dx^0dx^2+a_{03}dx^0dx^3+a_{12}dx^1dx^2+a_{13}dx^1dx^2+\\ a_{23}dx^2dx^3+a_{0123} dx^0dx^1dx^2dx^3\,.
\end{multline*}
In particular, every such spinor gives rise to a generalized complex structure of type $2$ and thus the projection $\rho_2$ of $\rho$ onto $\Omega^2(Y)$ defines an ordinary complex structure on $Y$. Once $\rho_2$ is chosen, no further conditions are imposed on $a_{0123}$. To summarize, $\rho$ is a spinor of an equivariant generalized complex structure on the gerby torus $Y$ if and only if $\rho_2\wedge\rho_2=0$ and $\rho_2\wedge \overline \rho_2>0$. This allows us to parametrize equivariant generalized complex structures on $(Y,\det)$ by the open subset of the hypersurface $a_{01}a_{23}-a_{02}a_{13}+a_{03}a_{12}=0$
in $\mathbb C \mathbb P^6$ on which 
\[
a_{01}\overline a_{23}-a_{02}\overline a_{13}+a_{03}\overline a_{12} + \overline a_{01}a_{23}-\overline a_{02} a_{13}+\overline a_{03}a_{12}>0\, . 
\]
If one wishes to quotient by infinitesimal generalized diffeomorphisms (as opposed to ordinary diffeomorphisms), or equivalently by orbits of the adjoint action with respect to the (twisted) Dorfman bracket, then all choices of $a_{0123}$ lead to isomorphic generalized complex structures. Therefore, we may set $a_{0123}=0$ and realize the moduli space as an open subset of the hypersurface in $\mathbb C\mathbb P^5$ which is the image of the Pl\"ucker embedding. This construction is a particular case of the realization of the moduli spaces of complex structures on tori as open subsets of Grassmannians described in \cite{atiyah-tori}.
\end{nolabel}
\begin{nolabel}
Without loss of generality, we may write $\rho=\eta_1\eta_2+a_{0123}\eta_1\eta_2\overline\eta_1\overline \eta_2$ for some invariant $1$-forms
\[
\eta_i = \sum_{j=0}^3\eta_{ij}dx^j\,;\quad i=1,2
\]
such that
\[
\eta_1\eta_2\overline\eta_1\overline \eta_2=dx^0dx^1dx^2dx^3=dx^0\det\,.
\]
Let $\{\xi_1,\xi_2,\overline\xi_1\,\overline \xi_2\}$ be a global frame for $TY\otimes \mathbb C$ with dual frame $\{\eta_1,\eta_2,\overline \eta_1,\overline \eta_2\}$. The $\sqrt{-1}$-eigenbundle of the generalized complex structure on $(Y,\det)$ associated with $\rho$ is the $C^\infty(Y)$-module generated by $\{\eta_1,\eta_2,\overline\xi_1-a_{0123}\overline\eta_2, \overline\xi_2+a_{0123}\overline \eta_1\}$. Therefore, the $\sqrt{-1}$-eigenspace of the $\brGe$-equivariant complex structure $\cJ$ on $\Xe$ induced by $\rho$ is the $C^\infty(\Xe)$-module generated by $\{\xi_1^*,\xi_2^*,\overline\xi_1-a_{0123}\overline \xi_2^*,\overline \xi_2+a_{0123}\overline \xi_1^*\}$, where $\xi_i^*=\sum_{j=0}^3\eta_{ij}\partial_{x_j^*}$ for $i=1,2$. Hence, $\cJ$ is uniquely defined by the left-invariant pure spinor
\[
\Omega_{\cJ}=\eta_1\eta_2(\overline\eta_1^*-a_{0123}\overline\eta_2)(\overline \eta_2^*+a_{0123}\overline\eta_1)
\]
on $\Xe$, where
\[
\eta_i^*=\sum_{j=0}^3\left(\eta_{ij}dx_j^*-\frac{1}{2}\varepsilon_{jkl}x^kdx^l\right)\,;\quad i=1,2\,.
\]
Here we are implicitly using the triviality of the generalized tangent bundles of $\Xe$ and $(Y,\det)$ to identify, abusing the notation, $dx^i$ with $\pi^*dx^i$, $\partial_{x^i}$ with $\pi_*\partial_{x^i}$ etc. Since the twisted Dorfman bracket on invariant sections of $(T\oplus T^*)Y$ matches the Lie bracket on left-invariant sections of $T\brXe$, we see immediately that complex structures on $\Xe$ corresponding to different values of $a_{0123}$ are isomorphic. 
\end{nolabel}

\begin{nolabel}
We may further assume that $\eta_1$ is of the form $\eta_1=dx^0+\eta$ with both $\eta$ and $\eta_2$ annihilated by $\partial_{x^0}$. In this notation, $\rho= dx^0\eta_2+\eta\eta_2+2a_{0123}dx^0\eta\eta_2\overline\eta_2$ and
\begin{equation}\label{eq:omegaJ}
\Omega_{\cJ}=\eta_2(dx_0^*-\eta^*-a_{0123}\overline\eta_2)(dx_0(\overline\eta_2^*-2a_{0123}\eta)+\eta\overline\eta_2^*)\,.
\end{equation}
T-duality in the topologically trivial direction $x^0$, exchanges even spinors on $(Y,\det)$ with odd spinors on the isomorphic gerby torus $(Y',\det)$ with coordinates $(x_0^*,x^1,x^2,x^3)$ via the fermionic Fourier transform \cite{bem-T-duality}
\[
\hat \rho=\int_{S^1}e^{dx_0^*dx^0}\rho\,.
\]
The identification of $(Y,\det)$ with $(Y',\det)$ shows that spinors of odd degree representing equivariant generalized complex structures on $(Y,\det)$ must be (up to scaling) of the form
\[
\hat \rho= \eta_2 - dx^0\eta\eta_2 + 2a_{0123} \eta\eta_2\overline\eta_2\,.
\]
A straightforward calculation shows that $\hat \rho$ defines an equivariant complex structure $\widehat \cJ$ on $X$ with pure spinor
\begin{equation}\label{eq:omegaJhat}
\Omega_{\widehat \cJ}= \eta_2(dx^0-\eta^*-a_{0123}\overline\eta_2)(dx_0^*(\overline\eta_2^*-2a_{0123}\eta)+\eta\overline\eta_2^*)\,.
\end{equation}
In other words, T-duality in the direction $x^0$ is reflected on $\Xe$ by exchanging $dx^0$ and $dx_0^*$ in \eqref{eq:omegaJ}. Since T-duality is an isomorphism of Courant algebroids \cite{cavalcanti-gualtieri-T-duality}, different choices of $a_{0123}$ lead to generalized complex structures that are intertwined by generalized diffeomorphisms. Therefore, T-duality is a symmetry of the moduli space of $(\cdot,\cdot)$-orthogonal equivariant complex structures on $\Xe$ which identifies the connected component corresponding to equivariant generalized complex structures of type $2$ on $(Y,\det)$ with the component corresponding to structures of type $1$. 

\end{nolabel}

\begin{nolabel}
In the patch where $a_{02}\neq 0$, we can further write $\eta_2=dx^2+\lambda$ with $\eta$ and $\lambda$ spanning the subspace generated by $dx^1$ and $dx^3$. As before, up to infinitesimal generalized diffeomorphisms, we may set $a_{0123}=0$. T-duality in the $x^2$-direction yields
\[
\hat{\hat \rho}=\int_{S^1}e^{(dx_2^*-\frac{1}{2}x^1dx^3+\frac{1}{2}x^3dx^1)dx^2}\hat \rho = e^{\lambda x_2^*+\eta x^0}
\]
thought of as a spinor of even degree on the nilmanifold $N$. Since the even de Rham cohomology of $N$ is generated by
\[
1\,,dx^0dx^1\,,dx^0dx^3\,,dx^1dx_2^*\,,(dx_2^*-\frac{1}{2}x^1dx^3+\frac{1}{2}x^3dx^1)dx^3\, ,dx^0dx^1dx_2^*dx^3\,,
\]
up to diffeomorphisms, all left-invariant symplectic structures on $N$ are of this form. If $a_{02}=0$, then $\hat{\hat \rho}$ defines an equivariant complex structure on $N$, possibly twisted by a B-field. If $a_{13}=0$, then there is no B-field and we obtain the description of the moduli of equivariant complex structures in the Kodaira surface of \cite{borcea-kodaira} (deformations of complex Kodaira surfaces as generalized complex structures are also described in \cite{poon} and \cite{brinzanescu-turcu}).

Regardless of the values of $a_{02}$ and $a_{13}$, each of these structures can be realized on $\Xe$ by swapping $dx^2$ and $dx_2^*-\frac{1}{2}x^1dx^3+\frac{1}{2}x^3dx^1$ in \eqref{eq:omegaJhat} for a suitable $\cJ$. Finally, just as in the case of the gerby torus $(Y,\det)$ described above, T-duality in the $x^0$ direction results into an isomorphism that interchanges the connected components of the moduli space of equivariant generalized complex structures on $N$. These structures can be realized on $\Xe$ by swapping $dx^2$ and $dx_2^*-\frac{1}{2}x^1dx^3+\frac{1}{2}x^3dx^1$ in the expression for $\Omega_{\cJ}$. To summarize, we have the following

\begin{thm*}
The following are isomorphic:
\begin{enumerate}
\item the moduli space of equivariant generalized complex structures on $(Y,\det)$;
\item the moduli space of equivariant generalized complex structures on $N$;
\item the moduli space of equivariant, $(\cdot,\cdot)$-orthogonal complex structures on $\Xe$;
\item the disjoint union of two copies of the moduli space of equivariant complex structures on $Y$.
\end{enumerate}
Each point in any of these isomorphic spaces determines an $N=2$ super-conformal  structure of central charge $c=12$ on $\cHe \simeq \cH_Y \simeq \cH_N$. 
\end{thm*}

\end{nolabel}

\bibliographystyle{unsrt}
\def\cprime{$'$}

\end{document}